\documentclass[12pt]{amsart}

\usepackage{amssymb}

\usepackage{enumerate}

\usepackage{graphicx}

\makeatletter
\@namedef{subjclassname@2010}{%
  \textup{2010} Mathematics Subject Classification}
\makeatother

\newtheorem{thm}{Theorem}[section]

\newtheorem{lem}[thm]{Lemma}

\newtheorem{claim}[thm]{Claim}

\theoremstyle{definition}
\newtheorem{defin}[thm]{Definition}

\newtheorem{exa}[thm]{Example}
\newtheorem{obs}[thm]{Observation}
\newcommand{\diam}{\operatorname{diam}}
\newcommand{\Gr}{\operatorname{Gr}}

\numberwithin{equation}{section}

\begin{document}


\baselineskip=17pt



\title[On the first homology of Peano Continua]{On the first homology of Peano Continua}

\author[Gregory R.\ Conner]{Gregory R.\ Conner}
\address{Mathematics Department\\
        Brigham Young University\\
        Provo, UT. 84602\\
        USA}
\email{conner@math.byu.edu}

\author[Samuel Corson]{Samuel Corson}
\address{Mathematics Department\\
1326 Stevenson Center\\
Vanderbilt University\\
Nashville, TN 37240\\
USA}
\email{samuel.m.corson@vanderbilt.edu}

\date{}

\begin{abstract}
We show that the first homology group of a locally connected compact metric space is either uncountable or is finitely generated.  This is related to Shelah's well-known result \cite{S} which shows that the fundamental group of such a space satisfies a similar criterion. We give an example of such a space whose fundamental group is uncountable but whose first homology is trivial, showing that our result doesn't follow from Shelah's. We clarify a claim made by Pawlikowski \cite{P} and offer a proof of the clarification.
\end{abstract}



\maketitle

\section{Introduction}
The classical Hahn-Mazurkiewicz Theorem states that a connected, locally connected compact metric space is the continuous image of an arc -- a \emph{Peano continuum}.  Shelah showed \cite{S} that if the first homotopy group, the \emph{fundamental group}, of a Peano continuum is countable then it is finitely generated.  This can be compared to our result:

\begin{thm}  The first homology group of a compact locally connected metric space is either uncountable or isomorphic to a direct sum of finitely many cyclic groups.
\end{thm}

If $X$ is a locally connected compact metric space then $X$ has only finitely many connected components, each of which is a Peano continuum and therefore path connected.  Then $H_1(X)$ is the direct sum of the first homology of each of the (path) components of $X$.  Thus it suffices to prove Theorem 1.1 for a Peano continuum.  For the sake of generality, we note that \emph{is countable} may be replaced by \emph{has cardinality less than the continuum} in both the result of Shelah and our result.


\begin{exa} We construct a space whose existence testifies that our result cannot follow immediately from Shelah's.  Let $X$ be a simplicial complex whose fundamental group is the alternating group $A_5$.  Recall that each element of $A_5$ is a commutator and thus $X^\mathbb{N}$, endowed with the product topology, is a Peano continuum whose fundamental group, ${A_5}^\mathbb{N}$, has the property that each of its elements is a commutator.  Whence $H_1(X^\mathbb{N})$, which is the abelianization of  ${A_5}^\mathbb{N}$  by the classical Hurewicz theorem, is trivial.
\end{exa}

\section{A construction}
We begin with the following lemma.
\begin{lem}\label{loctriv}  Let $X$ and $Y$ be Peano continua and $f: X \rightarrow Y$ be a mapping.  If there exists $\epsilon>0$ such that each loop of diameter less than $\epsilon$ is mapped under $f$ to a nulhomogolous loop, then $f_*(H_1(X))$ is a finite sum of cyclic groups.
\end{lem}

\begin{proof} Consider the composition of the Hurewicz map (which is surjective by the Hurewicz theorem) and the map $f_*$ and apply Lemma 7.6 in \cite{CC} which states that:

If $X$ is a connected, locally path connected separable metric space which is locally trivial with respect to $g: \pi_1(X) \rightarrow K$ then $g(\pi_1(X))$ is countable, and furthermore is finitely generated if $X$ is compact.\footnote{The referee has pointed out that there was a minor flaw in the proof of Lemma 7.6, which we correct here.  Select open covers $\{U_x\}$, $C_1$ and $C_2$ as in the original proof.  Choose $C$ to be an open refinement of $C_2$ which consists of open path connected sets.  If $X$ is compact we may take $C$ to be finite, and if $X$ is merely separable we may take $C$ to be countable (since separable metric spaces are Lindel{\"o}f).  Now apply Theorem 7.3 (2) of \cite{CC} to conclude that the image of $g$ is finitely generated (resp. countable) in case $X$ is compact (resp. separable).}

To finish, we need only remark that a finitely generated Abelian group is a finite sum of cyclic groups.
\end{proof}

Consider a Peano continuum $X$ with metric $d$ such that $H_1(X)$ is not finitely generated.  By Lemma \ref{loctriv}, letting $f:X \rightarrow X$ be identity, there exists a sequence $(f_n)_{n \in \mathbb{N}}$ of continuous mappings of $S^1$ to $X$, with $\diam(f_n) \leq \frac{1}{2^n}$, each of which is not nulhomologous.  By compactness of $X$, there exists a point $x \in X$ and a subsequence (without loss of generality, let the subsequence be the original sequence) such that the sequence of images $f_n(S^1)$ converges to $x$ in the Hausdorff metric.  By local path connectedness, we may assume that $x \in f_n(S^1)$ for all $n \in \mathbb{N}$, and by passing to a subsequence (without loss of generality, let the subsequence be the original sequence), we may assume again that $\diam(f_n) \leq \frac{1}{2^n}$ for all $n \in \mathbb{N}$.  Thus, each $f_k$ can be thought of as a loop $f_k: [0,1] \rightarrow X$ such that $f(0) = f(1) = x$ (this allows us to concatenate such mappings together).

Recall that the product space $\{0,1\}^{\mathbb{N}}$ is the Cantor set, where each factor $\{0,1\}$ is given the discrete topology.

Let $f_k^{1}$ be $f_k$ and $f_k^0$ be the constant loop at $x$. For $\alpha\in \{0,1\}^{\mathbb{N}}$, let $f^{\alpha} = f_0^{\alpha(0)} \ast f_1^{\alpha(1)} \ast f_2^{\alpha(2)} \ast \cdots$.  This can be thought of as the (pointwise) limit of a Cauchy sequence in the complete metric space $\mathcal{C}([0,1], X)$, the metric being the $\sup$ metric. Thus, each concatenation gives a continuous function from $S^1$ to $X$.  We define an equivalence relation on the Cantor set $\{0,1\}^{\mathbb{N}}$ as follows:  $\alpha \sim \beta$ iff $f^{\alpha}$ is homologous to $f^{\beta}$. 

\begin{defin} We say a space $Y$ is \emph{Polish} if it is completely metrizable
and separable (for example, the Cantor set $\{0,1\}^{\mathbb{N}}$.)  
If $Y$ is a Polish
space, we say that $A \subseteq Y$ is \emph{analytic} if there exists a
Polish space $Z$ and a closed set $D \subseteq Y \times Z$ such that
$A$ is the projection of $D$ in $Y$.  Analytic spaces are preserved
under continuous images and preimages, products, and under countable unions and intersections.   Closed sets are analytic.  
If $\sim$ is an analytic subset of $X \times X$ we say that \emph{$\sim$ is an analytic relation on $X$}.   A subspace of a topological space is \emph{perfect} if it is closed, nonempty and
has no isolated points.  A perfect subspace of a Polish space has the cardinality of
the continuum.
\end{defin}

\begin{lem}\label{littlelemma}  The space $\sim$ is an analytic equivalence relation on the Cantor set which has the property that if $\alpha$ and $\beta$ differ at exactly one point then $\alpha \nsim \beta$.
\end{lem}

\begin{proof}

Let $H$ be the space of all countinuous maps of $[0,1]$ to $X$ such that $\{0, 1\} \mapsto x$, with the metric on $H$ being the $\sup$ metric.  For each $n\in \mathbb{N}$ let $C_n: H^{2n} \rightarrow H$ be the map defined by mapping $(l_1, l_2, \ldots, l_{2n})$ to $$l_1\ast l_2\ast (l_1)^{-1}\ast (l_2)^{-1}\ast l_3\ast l_4\ast (l_3)^{-1}\ast (l_4)^{-1}\ast\cdots \ast l_{2n-1}\ast l_{2n}\ast (l_{2n-1})^{-1}\ast (l_{2n})^{-1}$$  Each such map $C_n$ is clearly continuous, and so the image $C_n(H^{2n})$ is an analytic subset of $H$.  Let $\mathcal{H}$ be the space of homotopies between loops, also under the $\sup$ metric.  Let $D \subseteq H^3 \times \mathcal{H}$ be defined by $D = \{(l_1, l_2, l_3, h): h$ homotopes $l_1$ to $l_2\ast l_3\}$.  The set $D$ is obviously closed in $H^3 \times \mathcal{H}$.  Then the set $D_n = (H\times H\times C_n(H^{2n}) \times \mathcal{H}) \cap D$ is analytic as an intersection of two analytic sets.  Letting $\sim_n$ be the projection of $D_n$ to $H \times H$, we see that $\sim_n$ is analytic.  Then $\bigcup_{n=0}^{\infty} \sim_n$ is also analytic.  Letting $F: \{0,1\}^{\mathbb{N}} \rightarrow H$ be given by $F(\alpha) = f^{\alpha}$, it is clear that $F$ is continuous.  Finally, $\sim = F^{-1}(\bigcup_{n=0}^{\infty} \sim_n)$, and so $\sim$ is analytic.

If $\alpha$ and $\beta$ differ exactly at $n \in \mathbb{N}$, then we know that $f_{0}^{\alpha(0)} \ast \cdots f_{n-1}^{\alpha(n-1)}$ is homotopic to $f_{0}^{\beta(0)} \ast \cdots f_{n-1}^{\beta(n-1)}$ and $f_{n+1}^{\alpha(n+1)} \ast f_{n+2}^{\alpha(n+2)} \cdots$ is homotopic to $f_{n+1}^{\beta(n+1)} \ast f_{n+2}^{\beta(n+2)} \cdots$ (the loops are in fact the same). Supposing that $\alpha \sim \beta$, we apply cancellations on the right and left so that $f_{n}^{\alpha(n)}$ is homologous to $f_{n}^{\beta(n)}$, so that $f_n$ is nulhomologous, a contradiction.
\end{proof}

\section{Proving Theorem 1.1}

Following the literature, there are two possible paths on which we might proceed to
finish the proof of our result.  Shelah \cite{S} and later Pawlikowski \cite{P} both argue that any equivalance relation on the Cantor set satisfying the conclusion of the previous lemma must contain a perfect set, which necessarily has the cardinality of the continuum.  However, this final portion of Shelah's seminal article is extremely terse and apparently uses a sophisticated technique related to forcing not generally available to the naive topologist.  For the sake of clarity and completeness we conclude by offering a complete and simplified discussion of Pawlikowski's method, which has the advantage of using terminology covered in a basic topology course. 

\begin{defin}
Recall that a set $A$ in a topological space $Y$ is \emph{nowhere dense} if its
closure has empty interior, is \emph{meager} if it can be written as a countable union of nowhere dense subsets of $Y$, and \emph{comeager} if $Y-A$ is meager. The Baire Category Theorem states that a complete metric space is not meager as a subset of itself.  A set $A \subseteq Y$ has \emph{the property of Baire} if there exists an open
set $O \subseteq Y$ such that $A \Delta O = (A \cup O) - (A\cap O)$
is meager (and say that $A$ is comeager in $O$.) Analytic spaces have the property of Baire. 
\end{defin}

To finish, we shall use the following consequence of the Kuratowski-Ulam Theorem:

\begin{lem}\label{KU}  If $Y$ is a Polish space and $A \subseteq Y \times Y$
is comeager in $U \times V$, with $U$ and $V$ open sets in $Y$, then
$\{y \in U: \{z \in V:(y,z) \in A\}$ is comeager in $V\}$ is
comeager in $U$.  Also, if $A \subseteq Y \times Y$ is meager, then
$\{y \in Y: \{z \in Y:(y,z) \in A\}$ is meager in $Y\}$ is comeager
in $Y$.

\end{lem}

For the uncountability of the number of classes of $\sim$ we offer a lemma similar to one in \cite{P} but whose proof is elementary.

\begin{lem}\label{Pawlikowski}  If $\sim$ is an equivalence relation on the Cantor set satisfying the conclusion of Lemma \ref{littlelemma} then $\sim$ is meager and has uncountably many classes.
\end{lem}

\begin{proof}  We show first that $\sim$ is meager.  Since $\sim$ is analytic it satisfies the property of Baire.
If $\sim$ were non-meager, then it would be comeager in a
neighborhood in $\{0,1\}^{\mathbb{N}} \times \{0,1\}^{\mathbb{N}}$
of the form $U \times V$, with $U$ and $V$ basic open sets in
$\{0,1\}^{\mathbb{N}}$.  By Lemma \ref{KU},

\begin{center}$A = \{\alpha \in U: \{\beta \in V: \alpha \sim \beta\}$ is
comeager in $V\}$

\end{center}

is comeager in $U$.  Let $n \in \mathbb{N}$ be greater than the
length of a generator of $U$.  Define $\Phi: U \rightarrow U$ by $\Phi(\alpha)(n) = 1-\alpha(n)$ and $\Phi(\alpha)(i) =
\alpha(i)$ for $i \neq n$.  Notice that $\Phi$ is a homeomorphism of
$U$ to itself.  Thus we can choose $\alpha \in A \cap \Phi(A)$, as $A \cap \Phi(A)$ is comeager.
Letting $\gamma  = \Phi(\alpha)$, we see that $\gamma$ and $\alpha$
differ only at $n$, so that $\alpha \nsim \gamma$.  From the
definition of $A$, we have comeagerly many $\beta \in V$ such that
$\alpha \sim \beta$.  Since $\gamma \in A$, the same can be said of
$\gamma$. Then for some $\beta$, we have $\beta \sim \alpha$ and $\beta
\sim \gamma$, implying $\gamma \sim \alpha$, a contradiction.

To see that $H_1(X)$ is uncountable we construct a set $Y \subseteq
\{0,1\}^{\mathbb{N}}$ of cardinality $\aleph_1$ such that for distinct
$\alpha$ and $\beta$ in $Y$ we have $\alpha \nsim \beta$.  For $\alpha \in \{0,1\}^{\mathbb{N}}$ we write
$\sim^{\alpha}$ as the equivalence class of $\alpha$.  As $\sim$ is
meager, we have by Lemma \ref{KU} that $J = \{\alpha \in
\{0,1\}^{\mathbb{N}}: \sim^{\alpha}$ is meager$\}$ is comeager in
$\{0,1\}^{\mathbb{N}}$. By transfinite induction define a
sequence $\{\alpha_i\}_{i < \omega_1}$ by picking $\alpha_j$ such
$\alpha_j \in J - \bigcup_{i<j} \sim^{\alpha_i}$, the Baire Category Theorem guaranteeing that this choice is always possible.  Let $Y = \{\alpha_i\}_{i < \omega_1}$.

\end{proof}

To see that there is in fact a continuum of pairwise non-homologous
loops, we invoke the following theorem of Mycielski
\cite{M}:

\begin{thm}\label{Myc}  Every meager relation on a perfect Polish space admits a perfect, pairwise non-related set.

\end{thm}

\section{A clarification in the non-compact setting}  In \cite{P}, a claim is stated without proof, which we state after first giving some definitions.

\begin{defin}   Let $\kappa$ be an infinite cardinal less than continuum.  We say a topological space $Y$ is \emph{$\kappa$-separable} if $Y$ has a dense subset of cardinality less than or equal to $\kappa$ and that $Y$ is \emph{$\kappa$-Polish} if $Y$ is completely metrizable and $\kappa$-separable.
\end{defin}

\begin{defin}  \cite{Mo}  If $Y$ is a Polish space, we say that $A \subseteq Y$ is \emph{$\kappa$-Suslin} if there exists a closed set $D \subseteq Y \times \kappa^{\mathbb{N}}$ such that
$A$ is the projection of $D$ in $Y$ (here $\kappa$ is given the discrete topology).
\end{defin}

\begin{obs}  Since any $\kappa$-Polish space is the continuous image of $\kappa^{\mathbb{N}}$ we may say $A\subseteq Y$ is $\kappa$-Suslin iff there exists a $\kappa$-Polish space $Z$ and a closed set $D \subseteq Y\times Z$ such that $A$ is the projection of $D$ in $Y$.  Thus $\aleph_0$-Polish spaces and $\aleph_0$-Suslin sets are precisely Polish spaces and analytic sets respectively.  Also, a metric space is $\kappa$-separable if and only if every open cover contains a subcover of cardinality at most $\kappa$. This latter condition is often called $\kappa$-Lindel{\"o}f.
\end{obs}

The aforementioned claim is the following \cite{P}:
\begin{claim}[Pawlikowski]  Let $\kappa$ be an infinite cardinal less than continuum.  Suppose that $X$ is a path connected locally path connected metric space which is $\kappa$-Lindel{\"o}f.  Then $\pi_1(X)$ is of cardinality $\leq \kappa$ or of cardinality continuum.
\end{claim}

We know of no proof of this claim using the methods of \cite{P}.  We state and prove a theorem with more hypotheses but which is nonetheless of interest.  Towards this, let $BP(\kappa)$ be the statement that all $\kappa$-Suslin sets have the property of Baire.  The claim $BP(\aleph_0)$ is simply true, as we noted earlier.  If one assumes extra set theoretic assumptions, for example Martin's Axiom, then $BP(\kappa)$ holds provided the successor cardinal $\kappa^+$ is less than continuum.  This follows from Theorem 2F.2 in \cite{Mo} combined with the consequence of Martin's Axiom that in a Polish space the collection of sets with the property of Baire is closed under unions of index less than continuum (see \cite{F}).  Martin's Axiom is known to be consistent with the standard axioms of Zermelo-Fraenkel set theory with the axiom of choice \cite {B}.

Our theorem is the following:

\begin{thm}  Assume $BP(\kappa)$.  Suppose that $X$ is a connected, locally path connected $\kappa$-Polish space.  Then $\pi_1(X)$ is either of cardinality $\leq \kappa$ or of cardinality continuum.
\end{thm}

\begin{proof}  There are two cases.  Suppose there exists a point $x$ near which we have arbitrarily small essential (non-nulhomotopic) loops.  By local path connectedness we have a sequence of essential loops $f_1, f_2, \ldots$ based at $x$ with $\diam(f_n)\leq 2^{-n}$.  Define a map from the Cantor set $\{0,1\}^{\mathbb{N}}$ to the space $L$ of continuous loops based at $x$ under the $\sup$ metric as in section 2 by letting $\alpha \mapsto f^{\alpha}$.  This map is continuous.  Define an equivalence relation $\approx$ on the Cantor set by letting $\alpha \approx \beta$ if and only if $f^{\alpha}$ is homotopic to $f^{\beta}$.

Let $D$ be the space of continuous mappings $H: [0,1]\times[0,1] \rightarrow X$ such that $H(s,0) = H(s,1) = x$, also under the $\sup$ metric.  Letting $\mathcal{D} = \{(f, g, H) \in L \times L\times D: (\forall t \in [0,1])(f(t) = H(0, t)$ and $g(t) = H(1, t))\}$, we have that the spaces $L,D$ and $L \times L\times D$ are all $\kappa$-Polish and $\mathcal{D}$ is closed in $L \times L\times D$.  Letting $\Gr$ be the graph of the map $(\alpha, \beta) \mapsto (f^{\alpha}, f^{\beta})$, the set $\Gr \times D$ is closed in $\{0,1\}^{\mathbb{N}} \times \{0,1\}^{\mathbb{N}} \times L \times L\times D$.  The relation $\approx$ is the projection of the closed set $(\{0,1\}^{\mathbb{N}}\times\{0,1\}^{\mathbb{N}}\times \mathcal{D}) \cap (\Gr \times D)$ to $\{0,1\}^{\mathbb{N}} \times \{0,1\}^{\mathbb{N}}$, and so is $\kappa$-Suslin.  By $BP(\kappa)$ the relation $\approx$ has the property of Baire.  If $\alpha, \beta \in \{0,1\}^{\mathbb{N}}$ differ at exactly one point, then $\alpha \not\approx \beta$ by the same proof as in Lemma \ref{littlelemma}.  By Lemma \ref{Pawlikowski} the set $\approx$ is meager, and so $\pi_1(X)$ is of cardinality continuum by Theorem \ref{Myc}.

Supposing that no such point exists, select for each $x\in X$ an open neighborhood $U_x$ such that any loop which lies in $U_x$ is nulhomotopic in $X$.  Since paracompactness follows from metrizability, select an open star refinement $\mathcal{U}_1$ of $\{U_x\}_{x\in X}$ that is locally finite.  By local path connectedness, let $\mathcal{U}_2$ be the open cover consisting of path components of elements of $\mathcal{U}_1$.  As $X$ is $\kappa$-separable and metric we may pick a subcover $\mathcal{U}\subseteq \mathcal{U}_2$ whose cardinality is less than or equal to $\kappa$.  Notice that the identity mapping $\pi_1(X) \rightarrow \pi_1(X)$ is 2-set simple relative to $\mathcal{U}$ (see \cite{CC} Definition 7.1).  Letting $\mathcal{N}$ be the nerve of $\mathcal{U}$ we have that $\pi_1(X)$ is a factor group of $\pi_1(\mathcal{N})$ by Theorem 7.3 (2) of \cite{CC}.  As $\mathcal{N}$ is a simplicial complex of at most $\kappa$ many vertices and $1$-cells, we are done.

\end{proof}

The comparable statement for first homology holds, via a similar argument.

\begin{thm}  Suppose $BP(\kappa)$.  Suppose that $X$ is a connected, locally path connected $\kappa$-Polish space.  Then $H_1(X)$ is either of cardinality $\leq \kappa$ or of cardinality continuum.

\end{thm}

\begin{proof}  The proof treats the two analogous cases as in the previous theorem.  Each element of $H_1(X)$ has a representative that is a mapping of a circle.  Suppose there exists a point $x$ near which there exist arbitrarily small mappings of $S^1$ that aren't nulhomologous. Then as in the proof of Theorem 1.1 we treat these mappings as paths that are based at $x$ by local path connectedness and select paths $f_n$ at $x$ that aren't nulhomologous such that $\diam(f_n) \leq \frac{1}{2^n}$, define $f^{\alpha}$ for each $\alpha \in \{0,1\}^{\mathbb{N}}$ and relations $\sim_n$ in the same way.  Each $\sim_n$ has the property of Baire since $\sim_n$ is $\kappa$-Suslin by a proof as in the previous theorem, and letting $\sim = \bigcup_{n \in \mathbb{N}} \sim_n$ we have that $\sim$ has the property of Baire, and $\alpha \sim \beta$ if and only if $f^{\alpha}$ is homologous to $f^{\beta}$.  The relation $\sim$ also enjoys the property that if $\alpha$ and $\beta$ differ at exactly one point then $\alpha \nsim \beta$.  By Theorem \ref{Pawlikowski} and Lemma \ref{Myc} we get continuum many classes.

If no such point exists, then for each $x \in X$ we have an open neighborhood $U_x$ such that any loop in $U_x$ is nulhomologous.  This gives an open cover $\{U_x\}_{x\in X}$ of $X$, which we refine as in the previous theorem to a cover $\mathcal{U}$ of cardinality at most $\kappa$ each of whose elements is path connected and such that given any mapping of $S^1$ to $X$ whose image lies entirely in the union of two elements of $\mathcal{U}$ is nulhomologous.  Let $g:\pi_1(X) \rightarrow H_1(X)$ be the Hurewitz map, which is onto.  This map $g$ is 2-set simple relative to $\mathcal{U}$ (\cite{CC} Definition 7.1), and so $g(\pi_1(X))$ is a factor group of the fundamental group of the nerve of $\mathcal{U}$, and so is at most of cardinality $\kappa$ by Theorem 7.3 (2) of \cite {CC}.

\end{proof}

Since $BP(\aleph_0)$ is simply true, we may state special cases of the above two theorems:

\begin{thm} \label{separable} Suppose that $X$ is a connected, locally path connected Polish space.  Then $\pi_1(X)$ is either countable or of cardinality continuum.  Also, $H_1(X)$ is either countable or of cardinality continuum.
\end{thm}

\subsection*{Acknowledgements}
The first author was supported by Simons Foundation Grant 246221.

\end{document}